\documentclass[12pt, english]{article}
\usepackage[english]{babel}
\usepackage[applemac]{inputenc}
\usepackage[T1]{fontenc} 

\usepackage{pgf}
\usepackage{tikz}
\usetikzlibrary{fit}

\usepackage{amsmath}
\usepackage{amssymb}
\usepackage{amsthm}

\DeclareRobustCommand{\stirling}{\genfrac\{\}{0pt}{}}

\theoremstyle{plain}
\newtheorem{Thm}{Theorem}
\newtheorem{Prop}[Thm]{Proposition}

\theoremstyle{definition}

\newtheorem{Def*}{Definition}

\newtheorem{Cor}[Thm]{Corollary}

\title{Slither code and the independence\\ number of a random tree}

\author{Johan Wästlund}% \\ Chalmers University of Technology}
\date{\today} 

\begin{document}

\maketitle

\begin{abstract}
We give a simple characterisation of the distribution of the independence number, and equivalently the matching number, of a random tree on $n$ labelled vertices chosen uniformly among the $n^{n-2}$ such trees: Roll an $n$-sided die repeatedly, and let $\alpha$ be the smallest number such that after $\alpha$ throws, at least $n-\alpha$ distinct numbers have occurred. Then $\alpha$ has the same distribution as the independence number, and $n-\alpha$ has the same distribution as the matching number. We obtain a similar characterisation of the path cover number. The proofs are bijective and based on modifications of the Prüfer code. 
\end{abstract}

\section{Introduction} \label{S:intro}
The starting point of this investigation was the following discovery:
\begin{Thm} \label{T:main}
Roll an $n$-sided die repeatedly and let $\alpha$ be the smallest number such that the first $\alpha$ throws produce at least $n-\alpha$ distinct numbers. Then $\alpha$ has the same random distribution as the maximum size of an independent vertex set in a uniformly random tree on $n$ labelled vertices.   
\end{Thm} 

For instance, if $n=4$ and we start by rolling two different numbers, we are done, while if we get the same number twice, we have to roll a third time. These events have probabilities $3/4$ and $1/4$ respectively, and indeed, of the sixteen trees on 4 labelled vertices, twelve are paths whose so-called independence number is 2, while the remaining four are stars with independence number 3. 

The clue that something like Theorem~\ref{T:main} should be true came from a study of random exponentially weighted edge-cover problems on complete bipartite graphs \cite{Cao, HesslerWastlund, W-arxiv}. This background will be clarified in \cite{W-preparation}, but in short, the distribution of the number of edges of an optimal solution can conjecturally be characterised in two different ways: On one hand it seems to be the same as the number of edges needed to cover a uniformly random spanning tree on the same underlying graph, and on the other hand there is an ``outer corner conjecture'' similar to the one in \cite{HesslerWastlund}, describing the distribution in terms of coupon collector processes. 

Conjectures like the one in \cite{HesslerWastlund} have been known since \cite{BCR} and so far resisted proof, but at least the two characterisations ought to be equal to each other! Finally we expected there to be a non-bipartite analogue, and that is Theorem~\ref{T:main}.

\section{Background}
The literature on random trees is vast, and several parameters like the independence number are quite well understood in the large $n$ limit. In particular, precise results are given for several random models of trees in \cite{BKP, Dadedzi, FHMN, Janson2}, see also \cite{BG, BES, KHK, MM1, Pittel, Wagner}. 
Still, as far as I know, the simple characterisation for ``finite $n$'' in Theorem~\ref{T:main} might be new. 

In this section we briefly outline some of the consequences of Theorem~\ref{T:main} and their relation to earlier work. We believe that Theorem~\ref{T:main} and the stronger Proposition~\ref{P:cards} in Section~\ref{S:reading} can shed light on some of the theorems derived analytically in for instance \cite{BKP, Pittel}, although it's less clear whether we can improve any of the results for large $n$.

Theorem~\ref{T:main} can be explored both from the point of view of random processes, and using analytic combinatorics. 
First, it's easy to argue informally that for large $n$, the dice game should last roughly $\rho\cdot n$ turns, where $\rho\approx 0.56714$, also known as the ``omega constant'', is the unique real solution to the equation $\rho = e^{-\rho}$: For $0<t<1$ and large $n$, if we roll an $n$-sided die $t n$ times, we expect the numbers we haven't yet seen to constitute a proportion of about $e^{-t}$ of the numbers. Since $\alpha$ is given by the time at which the numbers we haven't seen are just as many as the number of times we have rolled, we get the equation $t = e^{-t}$, whose solution is the omega constant.

Let $\alpha_n$ denote the independence number of a uniformly random labelled tree on $n$ vertices. The fact that $\alpha_n$ is concentrated around $\rho n$ has been known at least since \cite{MM1}, where Amram Meir and John Moon showed that    
\begin{equation} \label{expectation} 
\mathbb{E}(\alpha_n) = \sum_{k=1}^n \binom{n}{k} \left(\frac{-k}n\right)^{k-1} = \rho\cdot n + O(\sqrt{n}), 
\end{equation}
and moreover that 
\begin{equation} \label{omega} \frac{\alpha_n}{n} \overset{\rm p}\longrightarrow \rho.\end{equation}
This can also be understood in the framework of the local limit of a uniform labelled tree called $\rm PGW^\infty(1)$, see \cite{AB, AldousSteele, Grimmett}. 

The concentration result \eqref{omega} was strengthened by Boris Pittel \cite{Pittel} who established a ``central limit theorem'' for $\alpha_n$:
\begin{Prop} [Pittel 1999] \label{P:Pittel}
 \begin{equation} \label{pittel} \frac{\alpha_n - \rho \cdot n}{\sqrt{n}} \overset{\rm d}\longrightarrow \mathcal{N}(0, \sigma^2), 
 \end{equation}
where the variance is 
\begin{equation} \label{variance} \sigma^2 = \frac{\rho - \rho^2 - \rho^3}{(1+\rho)^2} \approx 0.025680.\end{equation}
\end{Prop}
More recently, analogous results were established by Cyril Banderier, Markus Kuba, and Alois Panholzer \cite{BKP} for the independence number and two other parameters in the more general setting of simply generated graphs. Some of their results are discussed in Sections~\ref{S:examples} and \ref{S:path}. 

We sketch a derivation of Proposition~\ref{P:Pittel} from Theorem~\ref{T:main}. Let $x$ be a real number, positive or negative. We want to estimate, for large $n$, the probability 
\[ \Pr\left(\alpha_n \leq \rho\cdot n + x\sqrt{n}\right). \]
We therefore run the dice game until we have seen $k = (1-\rho) n - x \sqrt{n} + O(1)$ distinct numbers. The crucial question is whether this happens in at most $\rho n + x\sqrt{n}$ throws.

By the standard analysis of the famous coupon collector problem, the number of throws we need before we get $k$ different numbers can be written $X_1+\dots+X_k$, where $X_i$ is the number of throws it takes us, once we have seen $i-1$ different numbers, before we see the next one. These terms are independent, each of a geometric distribution, and by elementary integral estimates,
\begin{equation} \label{e1} \mathbb{E}(X_1+X_2+\dots+X_k) = \rho\cdot n - \frac{x\cdot \sqrt{n}}{\rho} + O(1) 
\end{equation} 
and
\begin{equation} \notag \mathrm{var}(X_1+\dots+X_k) = n\cdot\left(\frac1{\rho}-1-\rho\right) + O\left(\sqrt{n}\right).\end{equation}
A simple calculation then shows that the target $\rho \cdot n + x\cdot\sqrt{n}$ is 
\begin{equation} \label{stdev1} x\cdot \frac{1+1/\rho}{\sqrt{1/\rho-1-\rho}} + O\left(\frac1{\sqrt{n}}\right)\end{equation}
standard deviations away from the mean \eqref{e1}.

Since we are not anywhere near the end of the coupon collector process, the third moments of $X_i$ are bounded, and we can invoke  the Berry-Esseen theorem to conclude that $X_1+\dots+X_k$ is approximately Gaussian. 
This means that the probability that the dice game ends in at most $\rho n + x \sqrt{n}$ throws is approximately the same as the probability that a standard normal variable takes a value of at most \eqref{stdev1}. Since this holds for every $x$, we conclude that \eqref{pittel} holds with an error of order $1/\sqrt{n}$ in the cumulative distribution function, and indeed the variance given by \eqref{variance} is the square reciprocal of the constant in \eqref{stdev1}.
 
Using analytic combinatorics, we can get even more precise results. Recall that the Stirling number of the second kind $\stirling{m}{k}$ is the number of partitions of $m$ elements into $k$ nonempty parts. 

\begin{Cor} \label{C:stirling}
The number of trees on $n$ labelled vertices and independence number $\alpha$ is 
\begin{equation} \label{enumeration} n^{n-\alpha-2}\cdot  \frac{n!}{\alpha!} \cdot \left(\stirling{\alpha}{n-\alpha} + \alpha \cdot \stirling{\alpha-1}{n-\alpha}\right)\end{equation}
\end{Cor}
\begin{proof}
We show how this follows from Theorem~\ref{T:main}. By the Cayley formula there are exactly $n^{n-2}$ trees on $n$ labelled vertices. Therefore we normalise by counting sequences of $n-2$ rolls of the die.  

There are two ways we can get the number $\alpha$ as the result of the dice game: Either there are exactly $n-\alpha$ different numbers occurring in the first $\alpha$ throws, or we ``overshoot'' by having $n-\alpha$ different numbers already in the first $\alpha-1$ throws and then getting a new number in the next throw. These two ways are counted by the first and second term of \eqref{enumeration} respectively.  
\end{proof}

A precise asymptotic estimate of Stirling numbers $\stirling{m}{k}$ was given by Nico Temme \cite{Temme} using a saddle-point method. In the regime $m/k \sim \rho/(1-\rho)$ of interest to us, it again leads to the conclusion of Proposition~\ref{P:Pittel} with the explicit constants, and moreover shows that the Gaussian approximation gives good estimates of each individual probability in the distribution. 

But this is the approach already taken in \cite{BKP, Pittel} starting from generating functions, so it's not clear whether \eqref{enumeration} can improve any of those results.  
Presumably one can show log-concavity of the expression \eqref{enumeration} as a function of $\alpha$, but we will not attempt any such analysis here.

\section{The game of Slither} \label{S:slither}
There is an archetypal family of two-person games on graphs that has probably been rediscovered several times, and whose different flavours are related to various combinatorial optimisation problems \cite{Trapping, Cao, HMM, GWGames, MS, L, MM2, PW, W-annals, W-arxiv}. The basic game was called \emph{Slither} in \cite{Anderson, Gardner}, and variants have been called \emph{Trapping}, \emph{Exploration}, and \emph{Percolation} games. 

Here we play the game on a rooted tree, where it is simpler than on a generic graph. We call it Slither as in \cite{Anderson, Gardner}, since \emph{slither code} seems like a suitable name for the operation that we introduce in Section~\ref{S:bijection}. In this section we describe the so-called \emph{normal} form of the game, which is the one that has to do with independence and matching numbers. In Section~\ref{S:path}
we discuss a ``twist'' related to another graph parameter, the path cover number.  

At the start of the game, a token is placed on the root vertex. The two players then take turns moving the token along an edge directed away from the root. The player who eventually moves to a leaf, so that the other one doesn't have any move options, is the winner.

In combinatorial games, a position is called a $P$-position, for \emph{previous} player win, if the player who just moved has a winning strategy, and an $N$-position, for \emph{next} player win, if the player to make the next move is winning. The vertices of a rooted tree can be classified in this way as $P$ or $N$, going from the leafs towards the root and following the rule that a vertex is a $P$-position if and only if none of its children is a $P$-position.  

The next proposition relates the game of Slither to the matching and independence numbers of the tree. Recall that the independence number is the largest size of a set of vertices where no two have an edge between them, and the matching number is the largest size of a set of edges of which no two meet in a vertex. In a bipartite graph, the König-Egerváry theorem states that the sum of the matching number and the independence number equals the number of vertices. For trees this is almost trivial, since maximum matchings and independent sets can be found by working ``greedily'' from the leafs. And finding this greedy independent set is actually what we do when we compute the set of $P$-positions of Slither.

\begin{Prop} \label{P:NPslither}
In a rooted tree, the number of $P$-positions in Slither is equal to the independence number, and the number of $N$-positions is equal to the matching number. 
\end{Prop}

\begin{proof}
The set of $P$-positions is clearly an independent set. Moreover, we can construct a matching $\mathcal{M}$ that pairs each $N$-position with a $P$-position among its children, simply by choosing such a child arbitrarily. Since no independent set can contain more than one vertex from each of the edges of $\mathcal{M}$, there is no larger independent set than the set of $P$-positions. Conversely, since every edge of a matching must use at least one $N$-position, there cannot be a matching larger than $\mathcal{M}$.
\end{proof}

\section{Slither code, a Prüfer type correspondence} \label{S:bijection}
In this and the next section we establish Theorem~\ref{T:main} through a bijection similar to the Prüfer code \cite{Moon, Prufer} between rooted trees on $n$ vertices and sequences of length $n-1$ taken from $1,\dots, n$. We first describe the bijection, that we call \emph{slither code}, and then analyse its properties in Section~\ref{S:reading}. 

\subsection{Computing the slither code of a rooted tree} \label{SS:computing}
Suppose we are given a tree on $n$ vertices labelled $1,\dots, n$ and rooted at any one of them. We regard all edges as directed away from the root. First we classify the vertices into $P$- and $N$-positions of Slither. Then we form an auxiliary sequence $(a_1,\dots, a_{n-1})$ that will constitute a permutation of the vertices other than the root. We think of our computation of this sequence as starting from $n-1$ empty slots where we eventually fill in numbers. We now repeatedly remove the leaf with the smallest label from the tree, and insert this label into the auxiliary sequence. The difference to the ordinary Prüfer correspondence is that we insert it in the leftmost empty slot if it's a $P$-position, and in the rightmost empty slot if it's an $N$-position (the classification into $N$ and $P$ is relative to the original tree and is not updated as we remove the leafs). 

This way we ``slither'' between left and right while the gap in the middle decreases. 
We keep going until we have filled the $n-1$ slots and only the root of the tree remains. Then we have completed the auxiliary sequence $(a_1,\dots, a_{n-1})$, and we get the slither code $(s_1,\dots, s_{n-1})$ of the tree by letting $s_i$ be the parent of $a_i$.

One interpretation of the slither code is that we sort the $n-1$ potential moves of the game from good to bad: If a vertex is a $P$-position, then moving to it from its parent is a good move and we place it in the leftmost free slot in our list, while if it's an $N$-position, moving to it is losing and we place it as far to the right as possible. 

\subsection{Retrieving a tree from the slither code} \label{SS.retrieving}
\begin{Prop}
The slither code provides a bijection between rooted trees on $n$ vertices labelled $\{1,\dots, n\}$, and sequences of length $n-1$ of numbers taken from $\{1,\dots, n\}$. 
\end{Prop}

\begin{proof}
Suppose we are given the slither code of a rooted tree. Plainly, the out-degree of every vertex is equal to the number of times it occurs in the slither code. We can therefore see from the code which vertices are leafs. This allows us to successively restore the tree from the leafs and up, in the order that the vertices were removed. 

At the typical stage of the process, we have drawn a set of edges of the tree, and we have only drawn edges to vertices whose entire subtree of descendants has already been restored. Subtracting the current out-degree of each vertex from the number of times it occurs in the slither code, we can find at least one vertex that already has all the edges out from it, but hasn't been assigned a parent. Of all such vertices, we find the one with the smallest label. That vertex must have been removed in the corresponding stage of computing the slither code from the tree. Since all its descendants have been restored, we can see if it's a $P$- or an $N$-position, and we therefore know from which end of the remaining slither code to read off its parent. We connect the parent to it, and that completes the stage. 
In the end, the vertex left without a parent becomes the root.
\end{proof}

\subsection{An example of inverting the slither code}
Suppose for instance that $n=10$ and the slither code of a tree is 
\[(3, 1, 4, 1, 5, 9, 2, 6, 5).\]
The numbers $1,\dots,10$ that don't occur are 7, 8, and 10. They are therefore the leafs of the original tree and thereby $P$-positions. The smallest of them is 7, and therefore 7 must have been removed first. It must then have been inserted in the leftmost slot in the auxiliary sequence, and therefore its parent is vertex~3. 

Since the number 3 doesn't occur anywhere else in the code, vertex 3 must have become a leaf when 7 was removed. At this stage the current leafs are 3, 8, and 10, and the smallest of them, 3, was removed next. We already know that 3 had a single child, 7, that was a $P$-position. Therefore 3 is an $N$-position, and we read off its parent 5 from the right end of the code.  

This time, since 5 occurs another time in the code, 5 is not yet a leaf. The leafs at this stage are therefore 8 and 10. The smallest of them is 8, and since it was a leaf already in the original graph, it's a $P$-position. We have already read one number from the left end of the code, so the parent of 8 is the next number, 1. 

We can fill in the auxiliary sequence below the slither code as we go. In the end we have: 

\begin{center}
\begin{tabular}{ccccccccc}
3 & 1 & 4 & 1 & 5 & 9 & 2 & 6 & 5\\
7 & 8 & 10 & 6 & 2 & 5 & 1 & 4 & 3
\end{tabular}
\end{center}
In the final step we assign vertex 9 as the parent of vertex 5, and 9, being left without a parent, must be the root. The restored tree looks as in Figure~\ref{F:piTree}.

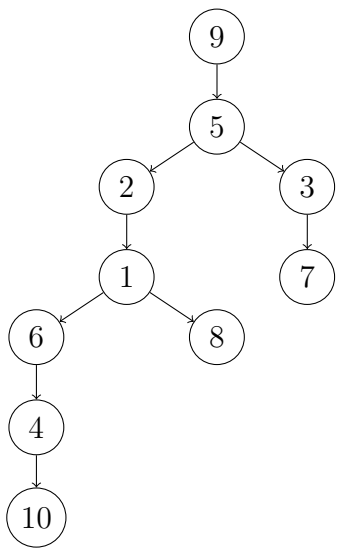
\begin{figure}
\begin{center}
\begin{tikzpicture}[xscale=1.2, yscale=0.8]

\node (9) [draw, circle] at (4,9) {$9$};
\node (5) [draw, circle] at (4,7.5) {$5$};
\node (2) [draw, circle] at (3,6.5) {$2$};
\node (3) [draw, circle] at (5,6.5) {$3$};
\node (1) [draw, circle] at (3,5) {$1$};
\node (7) [draw, circle] at (5,5) {$7$};
\node (6) [draw, circle] at (2,4) {$6$};
\node (8) [draw, circle] at (4,4) {$8$};
\node (4) [draw, circle] at (2,2.5) {$4$};
\node (10) [draw, circle, inner sep=3pt] at (2,1) {$10$};

\draw [->=stealth] (9) -- (5);
\draw [->=stealth] (5) -- (2);
\draw [->=stealth] (5) -- (3);
\draw [->=stealth] (2) -- (1);
\draw [->=stealth] (3) -- (7);
\draw [->=stealth] (1) -- (6);
\draw [->=stealth] (1) -- (8);
\draw [->=stealth] (6) -- (4);
\draw [->=stealth] (4) -- (10);

\end{tikzpicture}
\caption{The rooted tree with slither code $(3,1,4,1,5,9,2,6,5)$.}
\label{F:piTree}
\end{center}
\end{figure}
We can verify that we get the given slither code back by starting from this tree, in turn removing 7, 3, 8 etc.

\section{Reading off the matching and independence numbers from the slither code} \label{S:reading}
Next we show how the matching and independence numbers are manifested in the slither code. Since our aim is to prove Theorem~\ref{T:main}, the result we want to establish is the following:

\begin{Prop} \label{P:read1}
Suppose the rooted tree $T$ on $n$ vertices has the slither code $(s_1,\dots, s_{n-1})$. Let $\alpha$ be the smallest number such that the prefix $(s_1,\dots, s_\alpha)$ contains at least $n - \alpha$ distinct numbers. Then $\alpha$ is the independence number of $T$, and consequently $n-\alpha$ is the matching number. 
\end{Prop}

This will suffice for a proof of Theorem~\ref{T:main}, but while we're at it, we will extract some more information:

\begin{Prop} \label{P:read2}
Suppose that the slither code ``overshoots'' in the sense that the shorter prefix $(s_1,\dots, s_{\alpha-1})$ already contains $n-\alpha$ distinct numbers, and $s_\alpha$ is a new number so that $(s_1,\dots, s_\alpha)$ contains $n-\alpha+1$ distinct numbers. Then $s_\alpha$ is the root of the tree, and it's a $P$-position. The $\alpha$ vertices not occurring in the shorter prefix $(s_1,\dots, s_{\alpha-1})$ are the $P$-positions, and thereby constitute a maximum independent set.
\end{Prop}

\begin{Prop} \label{P:read3}
Suppose on the other hand that $(s_1,\dots, s_\alpha)$ contains exactly $n-\alpha$ distinct numbers. Then the root is an $N$-position. If $s_{\alpha+1}$ already occurs in the prefix $(s_1,\dots, s_\alpha)$, then that's the root. If on the other hand $s_{\alpha+1}$ occurs for the first time in slot $\alpha+1$, then $s_\alpha$ is the root. The $\alpha$ vertices not occurring in $(s_1,\dots, s_\alpha)$ are the $P$-positions and thereby a maximum independent set. 
\end{Prop}

\begin{proof} [Proof of Propositions \ref{P:read1}, \ref{P:read2}, and \ref{P:read3}]
Let $\alpha'$ be the number of $P$-positions of the rooted tree $T$. We want to show that $\alpha'=\alpha$ by showing that $\alpha'$ has the required property in terms of the slither code. 

Suppose first that the root is a $P$-position. Then the $P$-positions other than the root occupy the first $\alpha'-1$ slots of the auxiliary sequence. The $N$-positions are precisely the parents of these $\alpha'-1$ vertices, and there are $n-\alpha'$ of them. This means that the first $\alpha'-1$ slots of the slither code contain exactly $n-\alpha$ distinct symbols, which implies that $\alpha'=\alpha$. 

Moreover, since all the children of the root are $N$-positions, the last vertex removed when computing the slither code must have been an $N$-position, and it must have been inserted in slot $\alpha$ of the auxiliary sequence. Therefore $s_\alpha$ is the root, and it doesn't occur earlier in the slither code.

Suppose on the other hand that the root is an $N$-position. Then the $P$-positions are precisely the vertices in the first $\alpha'$ slots of the auxiliary sequence. The $N$-positions, including the root, must be precisely their parents, and there must be exactly $n-\alpha'$ of them. This means that the first $\alpha'$ slots of the slither code must contain exactly $n-\alpha'$ distinct numbers, and we can again conclude that $\alpha'=\alpha$. 

The last vertex to be removed must occur in the auxiliary sequence as either $a_\alpha$, the rightmost $P$-position, or $a_{\alpha+1}$, the leftmost $N$-position. Therefore either $s_\alpha$ or $s_{\alpha+1}$ is the root. If $s_{\alpha+1}$ doesn't occur earlier in the slither code, then it's a $P$-position and therefore not equal to the root. If on the other hand $s_{\alpha+1}$ is not a new symbol, but one that occurs earlier in the slither code, then it's an $N$-position. That means that the last $N$-position to be removed had a parent that was an $N$-position. That parent must then be the root. 
\end{proof}

This also completes the proof of Theorem~\ref{T:main}. But it actually gives more information: It shows that the dice game emulates the joint distribution of the independence number with the sequence of out-degrees. This is just Proposition~\ref{P:read1} combined with the simple observation that the out-degree of a vertex in a rooted tree is equal to the number of occurrences of its label in the slither code. Anyway, we formulate a strengthening of Theorem~\ref{T:main} to a version using playing cards instead of dice.

\begin{Prop} \label{P:cards}
Let $n$ be a positive integer and let $(d_1,\dots, d_n)$ be a sequence of nonnegative integers summing to $n-1$. Take a deck of $n-1$ cards, and for each $i=1,\dots, n$, label $d_i$ of the cards with the number $i$. Then shuffle the deck and deal the cards one by one and read off the number $\alpha$ in the same way as in the dice game. Then $\alpha$ has the same random distribution as the independence number of a rooted tree on $n$ vertices chosen uniformly among all trees with sequence of out-degrees $(d_1,\dots, d_n)$.   
\end{Prop}
 
\section{Examples: simply generated trees} \label{S:examples}
We give three examples of Proposition~\ref{P:cards} applied to so-called \emph{simply generated} random trees. For a survey of this concept, see \cite{Janson}. 

 \subsection{Full binary trees}
A so-called full binary tree is a rooted tree in which every vertex has out-degree 0 or 2. Such trees exist only when $n=2m+1$ is odd. To sample conveniently from them, we can specify that vertices $1,\dots, m$ have out-degree 2, and the remaining $m+1$ vertices are leafs. By adding an artificial leaf above the root, we can equivalently sample uniformly an unrooted tree where all vertices have degree 1 or 3, but let's stick to the rooted version.

We can generate such a tree by shuffling a deck of $2m = n-1$ cards containing two cards of each label $1,\dots, m$, and dealing a random slither code. We estimate, again informally, the expected independence number. We set $\alpha = t n$, and want to find $t$ such that we expect about $t n$ of the numbers $1,\dots, n$ not to occur in the first $t n$ slots of the slither code. The proportion of numbers that don't occur in the first $t n$ slots is approximately \[ \frac12 + \frac12\cdot (1-t)^2.\]
This is because half of them don't occur at all, and of those that do occur, some have both their occurrences in the last $(1-t)n$ slots. 
The equation 
\[ t = \frac12 + \frac12\cdot (1-t)^2\]
has the solution $t = 2 - \sqrt{2}$, and therefore our conclusion, at this point not completely rigorous, is that the independence number of a random tree with degrees 1 and 3 is concentrated at $2-\sqrt{2}$ times the number of vertices. 

But we can actually derive a combinatorial formula analogous to \eqref{enumeration}. Counting all labellings where the leafs get labels $m+1,\dots,2m+1$, the number of full binary trees with $2m+1$ vertices and independence number $\alpha$ is exactly
\begin{multline} \label{factorials} \frac{m!}{(\alpha-m-1)!(2\alpha-2m-1)!(4m-3\alpha+2)!}\cdot \frac{\alpha!(2m-\alpha)!}{2^{3\alpha-3m-2}}\\
+\frac{m!}{(\alpha-m-2)!(2\alpha-2m-2)!(4m-3\alpha+3)!}\cdot \frac{(\alpha-1)!(2m-\alpha)!}{2^{3\alpha-3m-4}}
 \end{multline}
 Just as in \eqref{enumeration}, the first term counts the trees where the root is an $N$-position, and the second term counts the cases of ``overshooting'' where the root becomes a $P$-position. 
 
 It's clear that we can get a central limit theorem analogous to Proposition~\ref{P:Pittel} by applying the Stirling approximation to all factorials and expressing $\alpha$ in terms of its deviation from $(2-\sqrt{2})\cdot 2m$. We can conveniently find the asymptotical variance of the independence number by computing the discrete logarithmic second derivative at $\alpha = (2-\sqrt{2})n$ of any one term of \eqref{factorials}. 
 The result is that to a first approximation the variance is
 \[ \left(17-12\sqrt{2}\right)\cdot m \sim \left(\frac{17}2 - 6\sqrt{2}\right) \cdot n \approx 0.014719 \cdot n.\] 
 The constant $17/2-6\sqrt{2}$ is obtained as the reciprocal of 
 \[ \frac1{\alpha-m}+\frac4{2\alpha-2m}+\frac9{4m-3\alpha} - \frac1{\alpha}-\frac1{2m-\alpha},\] evaluated (to get the correct scaling) at $m=1/2$ and $\alpha=2-\sqrt{2}$.

\subsection{Binary trees with left/right children}
Another example from the family of simply generated trees are binary trees where each vertex has potentially a left and a right child (and may have any one of them and not the other). We can sample from this family by shuffling a deck of $2n$ cards labelled $1_L, 1_R, 2_L, 2_R,\dots, n_L, n_R$, and dealing a random slither code of length $n-1$ (leaving $n+1$ cards unused).
This time the equation for $t\sim \alpha/n$ becomes
\[ t = (1-t/2)^2,\] 
which has the root $4-2\sqrt{3}$. This verifies the value denoted $\mu^{[N]}$ in Theorem~2 of \cite{BKP}, but we omit a rigorous analysis.

\subsection{Ordered (``plane'') trees}
Likewise, we can verify the value $\mu^{[N]} = (\sqrt{5}-1)/2$ given in \cite{BKP} for ordered rooted trees on $n$ vertices. These trees, also called Catalan trees, have a natural bijection to parenthetical expressions containing $n-1$ pairs of parentheses, with these pairs corresponding to the vertices other than the root. They too can be conveniently generated using a deck of cards. This time the translation to a slither code is not completely trivial, but for a discussion of plane trees and their uniform probability measure we refer to \cite[p. 28--29]{Janson}. 

We take a deck consisting of $n$ red cards numbered $1,\dots, n$, and $n$ unnumbered black cards. Again we shuffle and deal out all $2n$ cards. It's well-known that the sizes of the chunks (``blocks'') of consecutive red cards have the same distribution as the sequence of nonzero degrees of a Catalan tree. Therefore we can get a tree with labelled vertices but the structure of an ordered tree as follows: Let slot $i$ in the slither code have label $k$ whenever card $i$ has exactly $k-1$ black cards before it. Then the cards in a red block will correspond to slots with the same number in the slither code, which is essentially everything we need.

The independence number will now be the smallest number $\alpha$ such that the red cards labelled $1,\dots,\alpha$ represent at least $n-\alpha$ blocks.   

For a large $n$ analysis, the only thing we need to know is that the sizes of the red blocks have an asymptotically geometric distribution (of parameter $1/2$). The informal argument leads to the equation $t=1/(1+t)$ which has the root $(\sqrt{5}-1)/2$.

\section{Path cover and comply-constrain Slither} \label{S:path}
In this section we show that another graph parameter, the path cover number, is manifested in a so-called \emph{comply-constrain} twist of the game of Slither, and to a correspondingly modified version of the slither code. 

To clarify how this relates to our earlier results, recall again that a matching is a set of edges where no two meet at a vertex. This concept can be generalised by stipulating a max-degree, or ``capacity'' of each vertex. When the capacity is 2, we look for a maximum size edge set of which no more than two meet at any vertex. The capacity can be any number, and can even be specified individually for each vertex, but here we stick to uniform capacity~2. Such an edge set will be a collection of paths, and we therefore call it a \emph{path-collection}. Since the underlying graph is a tree, we can think of the optimisation problem in two equivalent ways: either we maximise the number of edges in a path-collection, or we minimise the number of components that the tree is divided into if the remaining edges are deleted. Thinking about it in the latter way is what leads to the concept of the \emph{path cover number}: covering the tree with a minimum number of disjoint paths, allowing paths consisting of a single vertex. Notice that the path cover number is to the maximum number of edges in a path-collection what the independence number is to the matching number - they add to the number $n$ of vertices.

\subsection{Comply-constrain Slither}
A ``twist'' that can be introduced for basically any two-person game, and that has actually been studied in its own right by combinatorial game-theorists, is called \emph{comply-constrain}. It can be arranged in several theoretically equivalent ways: One is that the player to move must select two move options, after which their opponent decides which of them should be played. Another is that the player who just moved can forbid one move option for the player who is about to move. In any case, we can think of the effect as forcing a player to play their second best move. 
The most ``logical'' winning condition, and the one we adopt here, is that it is permitted to win by playing to a vertex where the opponent has only one move option, and forbid it. 

The fact that a comply-constrain game provides information about path-collections is analogous to the situation in \cite{PW, W-arxiv}, where a similar comply-constrain game gives information about the traveling salesman problem. 

If we classify the vertices of a rooted tree into $P$-positions and $N$-positions with respect to comply-constrain Slither, then the rule is that a vertex is a $P$-position if and only if it has \emph{at most one} $P$-position among its children. Moreover, the $P$-positions are naturally divided into two categories $P_0$ and $P_1$ depending on whether they have no $P$-position among their children, or exactly one. So for every tree, the set of vertices is partitioned into the three subsets $P_0$, $P_1$, and $N$.  

The link between comply-constrain Slither and maximum path-collections is provided by the following parallel of Proposition~\ref{P:NPslither}:

\begin{Prop} \label{P:pcNP}
The maximum number of edges of a path-collection in a rooted tree is given by
\[ 2\cdot \left|N\right| + \left|P_1\right|.\]
\end{Prop}

\begin{proof}
Let's say that a set of edges is a \emph{strategic set} if it consists only of edges directed to $P$-positions, and it contains one such edge (the unique one) from each $P_1$-position, and two such edges from each $N$-position (here a choice might be possible). An interpretation is that a strategic set describes a strategy - the moves to play or forbid respectively.

Clearly a strategic set has exactly $2\cdot \left|N\right| + \left|P_1\right|$ edges, and no more than two meet at any vertex, in other words it's a path-collection.

The following two statements are now easily proved together by induction over subtrees: A strategic set maximises the number of edges in a path-collection, \emph{and} among all path-collections achieving this maximum, it minimises the number of edges at the root. 
\end{proof}

\subsection{Comply-constrain slither code}
Next we observe that the slither code works just as fine for the comply-constrain game. It's a different bijection, because the $P$-positions are not the same, but since we can still determine the status of a vertex from only its tree of descendants, the comply-constrain slither code too can be inverted. The arguments of Sections~\ref{SS:computing} and \ref{SS.retrieving} go through word by word.  

\subsection{Finding the path cover number from the slither code}
To see how Theorem~\ref{T:main} generalises to path-collections and the path cover number, let us first make a trivial reformulation of Proposition~\ref{P:read1}. 

\begin{Prop}
Let $(s_1,\dots, s_{n-1})$ be the normal slither code of a rooted tree $T$. Let $\beta$ be the smallest number such that the prefix $(s_1,\dots, s_\beta)$ contains at least $n-1-\beta$ distinct numbers. Then the matching number of $T$ is equal to the number of distinct numbers that occur in $(s_1,\dots, s_\beta)$. 
\end{Prop}

This might at first seem a bit confusing in relation to Proposition~\ref{P:read1}, since $\beta$ may or may not be equal to $\alpha$. But plainly, if there are exactly $n-1-\beta$ numbers in the prefix $(s_1,\dots, s_\beta)$, then $\alpha=\beta+1$ and the matching number of $T$ is $n-\alpha = n-1-\beta$. And if instead there are $n-\beta$ different numbers in $(s_1,\dots, s_\beta)$, then $\alpha=\beta$, and the matching number, this time $n-\beta$, is again equal to the number of distinct numbers in $(s_1,\dots, s_\beta)$. 

Now the comply-constrain analogue is fairly straightforward:

\begin{Prop}  \label{P:readPath}
Let $(s_1,\dots, s_{n-1})$ be the comply-constrain slither code of a rooted tree $T$. Let $\beta$ be the smallest number such that the prefix $(s_1,\dots, s_\beta)$ contains at least $n-1-\beta$ distinct doublets. Then the maximum number of edges in a path-collection in $T$ is the number of symbols that remain in $(s_1,\dots, s_\beta)$ if we discard the third and following occurrence of each number. 
\end{Prop}

\begin{proof}
Let $\alpha$ denote the number of $P$-positions. Suppose first that the root is a $P$-position. Then the remaining $P$-positions will occupy the leftmost $\alpha-1$ slots of the auxiliary sequence. The $N$-positions are precisely the vertices that have two or more children among the $P$-positions and therefore occur at least twice in the first $\alpha-1$ slots of the slither code. Since there are exactly $n-\alpha$ of them, it follows that $\beta=\alpha-1$, and that if we count the symbols in the prefix $(s_1,\dots, s_\beta)$, discarding the third and following occurrences of the same number, we get $2\left|N\right| + \left|P_1\right|$. 
 
Suppose on the other hand that the root is an $N$-position. Then the $P$-positions occur in the first $\alpha$ slots of the auxiliary code. In this case, the last $P$-position to be removed in the computation of the slither code must have had an $N$-position for parent. This parent is $s_\alpha$, which is therefore a number that occurs at least once before in the slither code. If the occurrence in slot $\alpha$ is the second occurrence of the number $s_\alpha$, then the prefix $(s_1,\dots, s_{\alpha-1})$ only contains $n-\alpha-1$ distinct doublets. This means that $\beta = \alpha$, and we get the correct count of $2\left|N\right| + \left|P_1\right|$ by looking at the prefix $(s_1,\dots,s_\alpha)$. If on the other hand the number $s_\alpha$ occurs two or more times already in $(s_1,\dots,s_{\alpha-1})$, then $\beta=\alpha-1$, but then the occurrence of $s_\alpha$ in slot $\alpha$ wouldn't be counted anyway, so we get the correct value of $2\left|N\right| + \left|P_1\right|$ even though we just count in the shorter prefix $(s_1,\dots, s_{\alpha-1})$.  
\end{proof}

\subsection{Asymptotics for uniform labelled trees}
We briefly comment on how the results of this section relate to those of \cite{BKP}. If we generate a random (comply-constrain) slither code by independent throws of a die, we expect that after $tn$ throws, the proportion of numbers not occurring is $e^{-t}$, and the proportion of numbers occurring exactly once is $te^{-t}$. Consequently, the numbers occurring at least twice are a proportion of around $1-(1+t)e^{-t}$. To find the $t$ corresponding to $\beta/n$, we solve the equation 
\[ t = (1+t)e^{-t}.\]
The solution doesn't seem to have a simple expression in terms of known functions and constants, but is approximately $t_0 \approx 0.80646$. At this point, we get asymptotically that $2\left|N\right| + \left|P_1\right|$ should be around 
\[ n\cdot (2 - (t_0+2) e^{-t_0}) \approx 0.74710\cdot n,\]
and therefore that the path cover number should be around $0.252899\cdot n$, in accordance with Table~1 of \cite{BKP}.

\section{Further generalisations}
The results of the previous section can readily be generalised to an arbitrary maximum capacity that we can denote by $b$ (as for so-called ``$b$-matchings''). In the capacity $b$ version of Slither, a player about to move must suggest $b$ different moves, and their opponent chooses which one should be played. If the player to move has fewer than $b$ move options, they lose the game. A strategic set is now a maximum set of edges of which no more than $b$ meet at any vertex. To read off the maximum number of edges in such a set (or the minimum number of such sets needed to cover the tree) from the correspondingly modified slither code, we let $\beta$ be the smallest number such that the first $\beta$ slots of the slither code contains at least $b$ occurrences each of $n-1-\beta$ distinct numbers. Then we count the symbols of this prefix, but only at most $b$ of each number.

We expect there to be natural generalisations to other graph parameters and also to related structures like hyper-trees (trees where ``edges'' connect more than two vertices), but those are beyond the scope of this paper.

\end{document}